\documentclass[12pt]{article}
\usepackage[utf8]{inputenc}
\usepackage{amsmath,amsfonts,amssymb,amsthm,mathrsfs}
\usepackage{graphicx}
\usepackage{xcolor}
\usepackage{hyperref}
\usepackage{cleveref}
\usepackage[margin=1in, top=0.5in]{geometry}

\theoremstyle{plain} 
\newtheorem{theorem}{Theorem}
\newtheorem{lemma}[theorem]{Lemma}
\newtheorem{proposition}[theorem]{Proposition}

\theoremstyle{definition} 
\newtheorem{remark}[theorem]{Remark}
\newtheorem{definition}[theorem]{Definition}
\newtheorem{example}[theorem]{Example}

\usepackage{tikz,pgfplots}
\usetikzlibrary{patterns,positioning,backgrounds}

\newcommand{\R}{\mathbb{R}}

\newcommand{\Pe}{\mathcal{P}}
\newcommand{\jj}{\mathtt{j}}

\newcommand{\norm}[1]{\left\|#1\right\|}
\DeclareMathOperator{\im}{im}

\title{Sufficient conditions for the Shrinking Wellness Lemma}
\author{Clemens Bannwart\footnote{Department of Physics, Mathematics and Computer Science, University of Modena and Reggio Emilia, Italy,
\texttt{clemens.bannwart@gmail.com}}}
\date{}

\begin{document}

\maketitle

\begin{abstract}
    The well groups were introduced by Edelsbrunner, Morozov, and Patel in \cite{Edelsbrunner_QuantTransRobustnessInters} to measure the robustness of geometric features of a function with respect to perturbations. Roughly speaking, the $r$-the well group measures the number of features that cannot be removed by perturbing the function by at most $r$. The Shrinking Wellness Lemma states that the rank of these groups decreases as $r$ increases. In the generality originally stated, it is wrong. We present a counterexample and give conditions under which the result holds. These conditions are general enough to cover most cases in which the well groups have been applied.
\end{abstract}

\section{Background and definitions}

The setup for defining and studying the well groups is the following. We need to have two topological spaces $X,Y$, a subset $A\subseteq Y$ and a continuous function $f\colon X \to Y$. Furthermore we need to fix a subset $\Pe \subseteq C(X,Y)$ of the set of all continuous functions $X\to Y$ and a metric $d_\Pe$ on $\Pe$. We assume that $f \in \Pe$.

\begin{remark}\label{rem:induced_metric}
If we have a metric $d_Y$ on $Y$, this induces an (extended) metric on $C(X,Y)$ (and thus on $\Pe$) by
\begin{equation*}
    d_\Pe(g,h) = \sup_{x \in X} d_Y(g(x),h(x)).
\end{equation*}
We put the word extended because in general the distance $d_\Pe$ can also take the value $\infty$. For the well groups, only those $g \in \Pe$ with $d_\Pe(f,g) < \infty$ will play a role.
\end{remark}

\begin{remark}\label{rem:most_interest}
A case of particular interest is the case where $X=\R^n$, $Y=\R^n$, $A=\{0\}$, $\Pe=C(X,Y)$, $d_Y$ is the Euclidean distance and $d_\Pe$ is induced by $d_Y$ as described above. This is for example the setting used in \cite{Chazal_CompWellVectorFields,WangHotz_Robustness2D,Wangetal_VisualizingRobustness2013}. Explicitly, this means that for $f,g \in \Pe$, 
\begin{equation*}
    d_\Pe(g,h) = \norm{g-h}_\infty = \sup_{x \in X} \norm{g(x)-h(x)}.
\end{equation*}
\end{remark}

In the paper \cite{Edelsbrunner_QuantTransRobustnessInters}, no conditions for the space $\Pe$ are given, however it is suggested that one may consider the space of all maps homotopic to $f$. As we see in the counterexample in the next section, we should put some conditions on $X,Y,A,f,\Pe,d_\Pe$ in order to make the Shrinking Wellness Lemma (SWL) true. 
We propose one possible set of conditions in \Cref{sec:tameness-proof}.
We now introduce the necessary concepts for defining the well groups.

\begin{definition}
Given $r\ge 0$, an \textbf{$r$-perturbation} of $f$ is a map $g \in \Pe$ such that $d_\Pe(f,g) \le r$. We write ${\Pe_r}$ for the set of all $r$-perturbations of $f$.
\end{definition}

\begin{definition}
We define the \textbf{well function} $f_A\colon X \to \R$ as follows:
\begin{equation*}
    f_A(x) := \inf \{r\ge 0 \mid \exists \text{ $r$-perturbation $g$ such that } g(x)\in A  \}. 
\end{equation*}
\end{definition}

\begin{remark}
In the case where we have a metric on $Y$, we can also consider these alternative definitions of the well function, 
\begin{equation*}
    f'_A(x) := d_Y(f(x),A)
\end{equation*}
and
\begin{equation*}
    f''_A(x) := \inf \{r \ge 0 \mid f(x) \in A_r \},
\end{equation*}
where $A_r := \{y \in Y \mid d_Y(y,A) \le r \}$ and $d_Y(y,A) = \inf_{a \in A} d_Y(y,a)$. In the context where the well groups have been applied mostly, the three definitions are the same, as we show in the next proposition. In \cite{Edelsbrunner_QuantTransRobustnessInters}, $f_A$ is used and in \cite{Chazal_CompWellVectorFields}, $f'_A$ is used, however there it takes an even simpler form because $A$ consists of just one point. In the less general situation studied in \cite{Chazal_CompWellVectorFields}, as well as in \cite{WangHotz_Robustness2D,Wangetal_VisualizingRobustness2013}, the different definitions agree, as we show in the next proposition.
\end{remark}

\begin{proposition}\label{prop:stuff}
Assume that we are given a metric $d_Y$ on $Y$.
\begin{enumerate}
    \item For all $x \in X$, we have $f'_A(x) = f''_A(x)$. \label{prop:stuff_item1}
    \item For all $r\ge 0$, we have ${f'_A}^{-1}([0,r]) = f^{-1}(A_r)$. \label{prop:stuff_item2}
    \item Assume that the metric $d_\Pe$ is induced from $d_Y$ as in \Cref{rem:induced_metric}. Then, for all $r\ge 0$, we have $f_A^{-1}([0,r]) \subseteq f^{-1}(A_r)$. \label{prop:stuff_item3}
    \item Assume further that $Y$ is a normed vector space and $d_Y$ is induced from the norm. Then, for all $r\ge 0$, we have $f^{-1}(A_r) \subseteq f_A^{-1}([0,r])$.\label{prop:stuff_item4}
    \item Under the conditions from \ref{prop:stuff_item4}, thus in particular under the conditions of \Cref{rem:most_interest}, we have $f_A = f_A' = f_A''$.\label{prop:stuff_item5}
\end{enumerate}
\end{proposition}

\begin{proof}
\hfill
\begin{enumerate}
\item We have 
\begin{align*}
    f_A'(x) \le r 
    &\iff \inf_{a \in A} d_Y(f(x),a) \le r \\
    &\iff \forall \epsilon > 0\ \exists a \in A\ d_Y(f(x),a) \le r+ \epsilon \hspace{51pt}
\end{align*}
and
\begin{align*}
    f''_A(x) \le r
    &\iff \inf \{s\ge 0 \mid f(x) \in A_s \} \le r \\
    &\iff \forall \epsilon >0\ f(x) \in A_{r+\epsilon} \\
    &\iff \forall \epsilon >0\ \inf_{a \in A} d_Y(f(x),a) \le r+\epsilon \\
    &\iff \forall \epsilon > 0\ \forall \delta >0\ \exists a \in A\ d_Y(f(x),a) \le r+\epsilon + \delta.
\end{align*}
Even though one more variable is used in the second statement, these two statements are equivalent, which shows that $f'_A(x)=f''_A(x)$.

\item follows directly from the definitions:
\begin{align*}
    x \in {f'_A}^{-1}([0,r])
    &\iff f'_A(x) \le r \\
    &\iff \inf_{a \in A} d_Y(f(x),a) \le r \\
    &\iff f(x) \in A_r \\
    &\iff x \in f^{-1}(A_r).
\end{align*}

\item Let $x \in f_A^{-1}([0,r])$. Then $f_A(x) \le r$, i.e. for all $\epsilon >0$ there exists an $(r+\epsilon)$-perturbation $h_\epsilon$ of $f$ such that $h_\epsilon(x) \in A$. Then 
\begin{equation*}
    d_Y(f(x),h_\epsilon(x)) \le \sup_{x' \in X} d_Y(f(x'),h_\epsilon(x')) = d_\Pe(f,h_\epsilon) \le r+\epsilon.
\end{equation*}
Since this holds for any $\epsilon>0$, we get that $d_Y(f(x),A) \le r$ and thus $f(x) \in A_r$, i.e. $x \in f^{-1}(A_r)$. 

\item The proof of this statement boils down to the following fact: Given $x \in X$ and $a \in Y$ with $d_Y(f(x),a) \le s$, then there exist $(s+\epsilon)$-perturbations $h_\epsilon$ of $f$ with $h_\epsilon(x)=a$, for arbitrarily small $\epsilon > 0$. Below we show that this fact is true when $Y$ is a normed vector space and $d_Y$ is the distance induced by the norm.

Let $x \in f^{-1}(A_r)$. Then $f(x) \in A_r$, i.e. $d_Y(f(x),A) \le r$. This means that for any $\epsilon >0$, there exists $a=a_\epsilon \in A$ such that $\norm{f(x)-a} \le r+\epsilon$.

Let $\Phi\colon Y \to Y$ be a continuous map such that
\begin{itemize}
    \item $\Phi$ is the identity outside of $B_{r+2\epsilon}(a)$.
    \item $\Phi \equiv a$ on $B_{r+\epsilon}(a)$.
    \item $\norm{\Phi(y)-y} \le r+2\epsilon\ \forall y \in Y$.
\end{itemize}
It is not too difficult to come up with an explicit description of such a map. In words, $\Phi$ contracts $B_{r+\epsilon}(a)$ to the point $a$, while leaving untouched everything outside of $B_{r+2\epsilon}(a)$. In between the two balls we have the condition that $\Phi$ maps the points not too far away from themselves.

Now consider $h=h_\epsilon := \Phi \circ f$. We have that 
\begin{equation*}
    \norm{h-f} = \norm{\Phi \circ f - f} = \sup_{x' \in X} \norm{\Phi(f(x'))-f(x')} \le \sup_{y \in Y} \norm{\Phi(y)-y} \le r+2\epsilon.
\end{equation*}
Therefore, since $h(x)=a \in A$ and $\norm{h-f} \le r+2\epsilon$ is possible for all $\epsilon>0$, we have that $f_A(x)\le r$ and thus $x \in f_A^{-1}([0,r])$.

\item The equality $f'_A=f''_A$ holds in general by \ref{prop:stuff_item1}. By \ref{prop:stuff_item3} and \ref{prop:stuff_item4}, the sublevelsets of $f_A$ are given by $f^{-1}(A_r)$, which on the other hand are equal to the sublevelsets of $f'_A$ by \ref{prop:stuff_item2}. Thus $f_A$ and $f'_A$ have the same sublevelsets and must thus be equal. \qedhere
\end{enumerate}
\end{proof}

\begin{definition}
Note that for an $r$-perturbation $g$ of $f$, we have  $g^{-1}(A) \subseteq f_A^{-1}([0,r])$. We denote the homomorphism in $0$-homology induced by this inclusion by
\begin{equation*}
    \jj_g^r\colon H_0(g^{-1}(A)) \to H_0(f_A^{-1}([0,r])).
\end{equation*}
We use homology with coefficients in a fixed field, which we hide from the notation. Note that in \cite{Edelsbrunner_QuantTransRobustnessInters}, $r$ is hidden from the notation, i.e. it is just written $\jj_g$ instead of $\jj_g^r$.
\end{definition}

\begin{remark}
Note that in \cite{Edelsbrunner_QuantTransRobustnessInters}, the definitions are given by using the direct sum of the homology groups in all degrees, not just degree $0$. We simplify the definition since in the case we are interested in, only $0$-homology has interesting well groups. Also in \cite{WangHotz_Robustness2D,Wangetal_VisualizingRobustness2013}, only $0$-homology is considered. In \cite{Chazal_CompWellVectorFields}, $n$-homology is considered, where $n$ is the dimension of the euclidean space on which the vector field is defined.
\end{remark}

\begin{definition}
The \textbf{well groups} of $f$ (w.r.t. $(\Pe,d_\Pe)$) are defined as
\begin{equation*}
    U(r) := \bigcap_{g \in {\Pe_r}} \im(\jj_g^r) \subseteq H_0(f_A^{-1}([0,r])),
\end{equation*}
for any $r\ge 0$.
\end{definition}

The homology classes of $f^{-1}(A)$ can be thought of as geometric features of the map $f$ and they map surjectively onto the homology classes of $f_A^{-1}([0,r])$. The well group $U(r)$ then describes those features that cannot be avoided by any $r$-perturbation of $f$.
We now need just one more definition before we can state SWL.

\begin{definition}
For $r \le s$, the inclusion $f_A^{-1}([0,r]) \subseteq f_A^{-1}([0,s])$ induces a map in homology, which we denote by
\begin{equation*}
    f_r^s\colon H_0(f_A^{-1}([0,r])) \to H_0(f_A^{-1}([0,s])).
\end{equation*}
\end{definition}

\begin{lemma}[Shrinking Wellness Lemma]
For $r\le s$ we have 
\begin{equation*}
    U(s) \subseteq f_r^s(U(r)).
\end{equation*}
\end{lemma}

There is a very easy wrong proof for SWL, which goes as follows:

\begin{align*}
    U(s) = \bigcap_{g \in \Pe_s} \im(\jj_g^s)
    &\subseteq \bigcap_{g \in {\Pe_r}} \im(\jj_g^s)\\
    &= \bigcap_{g \in {\Pe_r}} \im(f_r^s \circ \jj_g^r)\\
    &= \bigcap_{g \in {\Pe_r}} f_r^s(\im(\jj_g^r))\\
    &{\color{red}=} f_r^s\left(\bigcap_{g \in {\Pe_r}} \im(\jj_g^r) \right)
    = f_r^s(U(r)).
\end{align*}
The problem is the step of the first equality on the last line (drawn in red), which can fail to hold in some cases. In general, given linear subspaces $A,B$ of some vector space $V$ and a linear map $F\colon V \to W$, we have that $F(A)\cap F(B) \supseteq F(A\cap B)$, while the other inclusion is not always true. Note however that equality holds if $F$ is injective. Applying this to our situation proves SWL in the case where $f_r^s$ is injective. The condition that $f_r^s$ is injective means that no two path-components merge when going from $f_A^{-1}([0,r])$ to $f_A^{-1}([0,s])$.

\section{A counterexample}\label{sec:counterexample}

In the following example we describe a situation in which SWL does not hold, i.e. there are two radii $r<s$ such that $U(s)$ has a larger dimension than $U(r)$.

\begin{example}
    Let $X=[-3,3]$, $Y=\R$, $A = \{-2,2\}$, and denote by $f\colon X \to Y$ the inclusion map. Let $\Pe \subseteq C(X,Y)$ be the set of order-preserving isometric embeddings. By $d_Y$ we denote the usual distance on $Y$. This means that we allow only shifts to the right or to the left as our perturbations, no deformations. The restriction of the supremum norm from $C(X,Y)$ to $\Pe$ yields a distance on $\Pe$. 

    Note that $A_r = \{y \in Y \mid d_Y(y,A) \le r \}$ is given by $[-2-r,-2+r] \cup [2-r,4+r]$. Also, $f^{-1}(A_r) = A_r \cap X$. It has two connected components for $0 \le r < 2$ and one for $2\le r < \infty$. Thus
    
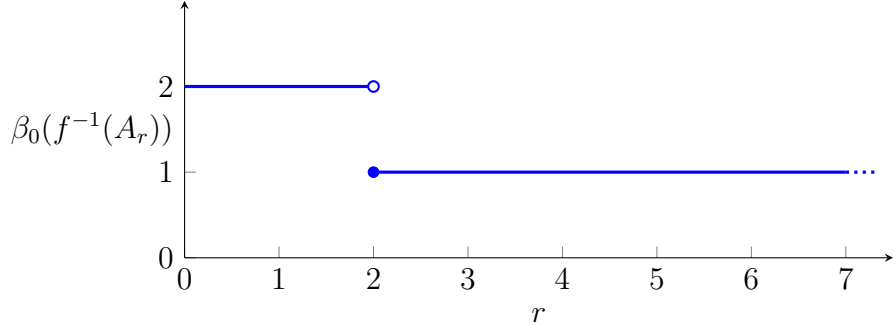
\begin{center}
\begin{tikzpicture}
  \begin{axis}[
      width=11cm, height=5cm,
      xmin=0, xmax=7.5,
      ymin=0,  ymax=3,
      axis lines=left,            
      xtick={-1,0,1,2,3,4,5,6,7},     
      ytick={0,1,2},            
      xlabel={$r$}, ylabel={\rotatebox{-90}{$ \beta_0(f^{-1}(A_r))$}}, 
      tick align=inside,
      enlargelimits=false
    ]
    \addplot[blue, very thick, domain=0:2] {2};

    \addplot[blue, only marks, mark=*, mark options={draw=blue, fill=white}, mark size=2pt, thick]
        coordinates {(2,2)};

    \addplot[blue, very thick, domain=2:7] {1};

    \addplot[blue, only marks, mark=*, mark size=2pt]
        coordinates {(2,1)};

    \addplot[blue, very thick, dotted, domain=7:7.35] {1};
  \end{axis}
\end{tikzpicture}
    \end{center}
    How does the rank of the well group $U(r)$ evolve as $r$ increases?

    For $0\le r \le 1$, any $r$-perturbation still hits both points of $A$, thus $U(r)$ has rank $2$. 

    For $1<r<2$, an $r$-perturbation of $f$ can avoid one of the two points of $A$, by shifting either to the left or to the right by a distance $r$, but not both at the same time. Therefore $U(r)=0$.

    For $2\le r \le 5$, we have $f^{-1}(A_r)=X$, so it has now only one connected component. Also, no $r$-perturbation can avoid both points of $A$. Therefore $U(r)$ is of rank $1$.

    For $r>5$, there exists an $r$-perturbation $h$ of $f$ with $h^{-1}(A)=\emptyset$ (i.e. $h$ avoids both points of $A$), thus $U(r)=0$.

    In conclusion, we get the following rank of the well groups:
    
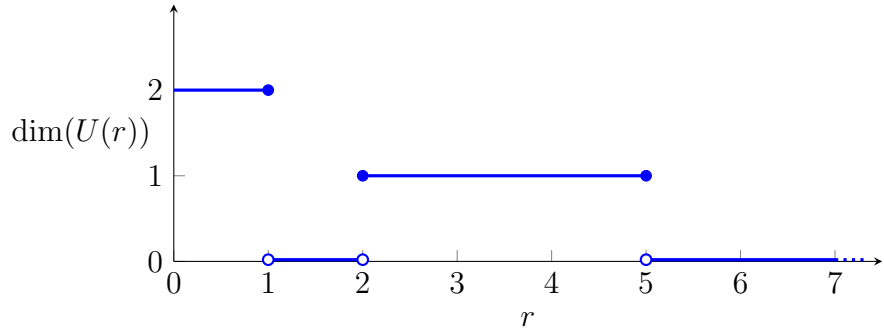
\begin{center}
\begin{tikzpicture}
  \begin{axis}[
      width=11cm, height=5cm,
      xmin=0, xmax=7.5,
      ymin=0,  ymax=3,
      axis lines=left,            
      xtick={-1,0,1,2,3,4,5,6,7},     
      ytick={0,1,2},            
      xlabel={$r$}, ylabel={\rotatebox{-90}{$ \dim(U(r))$}}, 
      tick align=inside,
      enlargelimits=false
    ]
    \addplot[blue, very thick, domain=0:1] {2};
    \addplot[blue, only marks, mark=*, mark size=2pt]
        coordinates {(1,2)};
        
    \addplot[blue, very thick, domain=1:2] {0.02};
    \addplot[blue, only marks, mark=*, mark options={draw=blue, fill=white}, mark size=2pt, thick]
        coordinates {(1,0.02)};
    \addplot[blue, only marks, mark=*, mark options={draw=blue, fill=white}, mark size=2pt, thick]
        coordinates {(2,0.02)};

    \addplot[blue, very thick, domain=2:5] {1};
    \addplot[blue, only marks, mark=*, mark size=2pt]
        coordinates {(2,1)};
    \addplot[blue, only marks, mark=*, mark size=2pt]
        coordinates {(5,1)};

    \addplot[blue, very thick, domain=5:7] {0.02};
    \addplot[blue, only marks, mark=*, mark options={draw=blue, fill=white}, mark size=2pt, thick]
        coordinates {(5,0.02)};
    \addplot[blue, very thick, dotted, domain=7:7.35] {0.02};
  \end{axis}
\end{tikzpicture}
    \end{center}
    Thus, we observe that the well groups are not shrinking when going from values $1<r<2$ to values $2\le s \le 5$.
\end{example}

Note that here the space $\Pe$ of perturbations is very small (allowing only shifting to the right or to the left), which is used to produce the non-shrinking in this example. It is not the same context as in which the well groups are applied mostly, but it shows that some extra conditions need to be formulated to prove SWL. 
We do this in the next section.

\section{Notions of tameness and proof of SWL}\label{sec:tameness-proof}

We are trying to find some conditions which are strong enough to make SWL true, but also weak enough to be broadly applicable.

The following definition is used in \cite{Edelsbrunner_QuantTransRobustnessInters}:

\begin{definition}
The function $f$ is called \textbf{admissible}, if $f^{-1}(A) = f_A^{-1}(0)$ has a finite rank $0$-homology group.
\end{definition}

This condition is strictly weaker than $f_A$ being tame or q-tame:

\begin{definition}
A function $F\colon X \to \R$ is called 
\begin{itemize}
    \item \textbf{tame} if all the $0$-homology groups of the sublevelsets $F^{-1}((-\infty,r])$ have finite rank,
    \item \textbf{q-tame} if all the maps in $0$-homology that are induced by the inclusions $F^{-1}((-\infty,r]) \subseteq F^{-1}((-\infty,s])$ have finite rank for $r < s$.
\end{itemize}
\end{definition}

The following theorem gives sufficient conditions for SWL to hold.

\begin{theorem}\label{thm:tame_SWL}
If $Y=\R^n$, $d_Y$ denotes the euclidean distance, $A=\{a\} \subseteq Y$ consists of a single point, $\Pe = C(X,Y)$ is endowed with the supremum norm $d_\Pe$, and $f\colon X \to Y$ is a function such that $f_A$ is tame, then SWL holds.
\end{theorem}

The proof is based on the following two lemmata.

\begin{lemma}\label{lem:extend_pert}
Let $X$ be a topological space, let $Y$ be a normed vector space and $A=\{0\}$. Let $\Pe=C(X,Y)$ and denote by $d_\Pe$ the supremum norm on $\Pe$. Let $g,g'$ be two $r$-perturbations of $f$ and let $C,C' \subseteq X$ be two closed sets, 
such that $g$ has no zero on $C$ and $g'$ has no zero on $C'$. Then there exists an $r$-perturbation $h$ of $f$ such that $h$ has no zero on $C \cup C'$.
\end{lemma}

\begin{proof}
By Urysohn's Lemma (which we can apply since $Y$ is a normed vector space, so in particular it's metrizable and thus normal) there exists a function $\varphi\colon Y \to [0,1]$ with $\varphi \equiv 1$ on $C$ and $\varphi \equiv 0$ on $C'$. We define $h:= \varphi\cdot g + (1-\varphi)\cdot g'$ and we have to show two things.

\underline{$h$ is an $r$-perturbation:} For any $x\in X$, we have 
\begin{align*}
    \norm{f(x)-h(x)} &= \norm{\varphi(x)(f(x)-g(x)) + (1-\varphi(x))(f(x)-g'(x))} \\
    &\le \varphi(x)\norm{f(x)-g(x)}+(1-\varphi(x))\norm{f(x)-g'(x)}\\ &\le \varphi(x)\cdot r + (1-\varphi(x))\cdot r = r.
\end{align*}

\underline{$h$ has no zero on $C\cup C'$:} On $C$ we have $h\equiv g$, thus $h$ has no zero on $C$. On $C'$ we have $h\equiv g'$, thus $h$ also has no zero on $C'$. 
\end{proof}

The idea of the next lemma is to give a sufficient condition for SWL to hold. In words, this condition demands that for every $r\ge 0$, there exists an $r$-perturbation $h$, such that $h$ has no zeros on every path-component where it's possible, i.e. for any path-component $C$ of $f_A^{-1}([0,r])$, if there exists an $r$-perturbation with no zero on $C$, then $h$ has no zero on $C$. Let us define this more formally now.

\begin{definition}
An $r$-perturbation $h \in \Pe_r$ is called \textbf{$r$-minimizing}, if
\begin{equation*}
    \bigcap_{g \in {\Pe_r}} \im(\jj_g^r) = \im(\jj_h^r).
\end{equation*}
\end{definition}

\begin{lemma}\label{lem:property_SWL}
Assume that $\Pe$ has the property that for every $r\ge 0$, there exists an $r$-minimizing $r$-perturbation. Then SWL holds.
\end{lemma}

\begin{proof}
This property is useful because it makes work the wrong proof that we presented earlier. This is done as follows: Let $r \le s$ and let $h \in {\Pe_r}$ be an $r$-minimizing $r$-perturbation. Then we get that

\begin{align*}
    U(s) = \bigcap_{g \in \Pe_s} \im(\jj_g^s)
    &\subseteq \bigcap_{g \in {\Pe_r}} \im(\jj_g^s)\\
    &= \bigcap_{g \in {\Pe_r}} \im(f_r^s \circ \jj_g^r)\\
    &= \bigcap_{g \in {\Pe_r}} f_r^s(\im(\jj_g^r))\\
    &{\color{green}\subseteq} f_r^s(\im(\jj_h^r))\\
    &{\color{green}=}f_r^s\left(\bigcap_{g \in {\Pe_r}} \im(\jj_g^r) \right)
    = f_r^s(U(r)),
\end{align*}
where inclusion on the second to last line and the first equality on the last line (both in green) replace the red equality from before.
\end{proof}

\begin{proof}[Proof of \Cref{thm:tame_SWL}]
Without loss of generality we may assume that $A=\{0\}$, since otherwise we can just replace $f$ by a translation of $f$.
The strategy of the proof is to use tameness together with \Cref{lem:extend_pert} to show that $r$-minimizing perturbations exist for all $r\ge 0$.

Let $r \le s$. By tameness, $H_0(f_A^{-1}([0,r]))$ has finite rank, i.e. $f_A^{-1}([0,r])$ has finitely many path-components. Let us label them as $C_1,\ldots,C_m$ in such a way that for $1\le i \le k$ there exists an $r$-perturbation $g_i$ with no zero on $C_i$ and for the components $C_{k+1},\ldots,C_m$, no such $r$-perturbations exist. 

By iteratively applying \Cref{lem:extend_pert}, we can get one $r$-perturbation with no zero on $C_1 \cup \cdots \cup C_k$. We first define $h_1 := g_1$. Then, in the $i$-th step, if we have already found an $r$-perturbation $h_i$ with no zero on $C_1\cup \cdots \cup C_i$, then we can apply \Cref{lem:extend_pert} to $h_i$ and $g_{i+1}$ to get an $r$-perturbation $h_{i+1}$ with no zero on $C_1\cup \cdots \cup C_{i+1}$. After $d$ steps, $h:=h_k$ is the desired $r$-perturbation with
\begin{equation*}
    \bigcap_{g \in {\Pe_r}} \im(\jj_g^r) = \im(\jj_h^r).
\end{equation*}
This shows that the condition from \Cref{lem:property_SWL} is satisfied and thus SWL holds.
\end{proof}

\begin{remark}
If one considers the counterexample in \Cref{sec:counterexample}, then one can see that the condition for \Cref{lem:property_SWL} is not met. Specifically, there exists no $r$-minimizing $r$-perturbation for $1<r<2$. 
\end{remark}

\begin{remark}
    One might argue that the setting from \Cref{rem:most_interest} is the one of most interest, since \cite{Chazal_CompWellVectorFields,WangHotz_Robustness2D,Wangetal_VisualizingRobustness2013} are all treating this case. In \cite{Edelsbrunner_StabilityAppCont2Manifold}, a slightly different case is treated, where $X$ can be any 2-dimensional manifold and $A$ can be any point in $\R^2$. Adding the tameness assumption, it follows from \Cref{thm:tame_SWL} that SWL holds in these situations. 
\end{remark}

\section{Conclusion and open questions}

We presented an example that shows that SWL does not hold under the most general conditions and we gave a set of sufficient conditions under which SWL holds. These are general enough to be applicable in most cases of interest, but it seems plausible that there might be weaker conditions which would suffice to prove SWL. For example, it might be possible to weaken the condition of tameness of $f_A$ to q-tameness, or even just admissibility of $f$. Another direction of investigation could be to formulate precise conditions on $\Pe$ and $d_\Pe$, instead of just demanding that $\Pe=C(X,Y)$ and $d_\Pe$ is the distance induced by the supremum norm.

\paragraph{Acknowledgements.} I would like to thank my supervisor Claudia Landi as well as Dmitriy Morozov for helpful discussions on this topic. 
This work was carried out under the auspices of  INdAM-GNSAGA and within the activities of ARCES (University of Bologna).

\bibliographystyle{amsplain}
\bibliography{Bibliography}

\end{document}